\documentclass[11pt]{elsarticle}

\usepackage{amsmath,amssymb,amsthm}
\usepackage{xcolor} 
\usepackage{epsfig}
\usepackage{graphicx}
\usepackage[english]{babel}
\usepackage{algorithm,algorithmic}

\newtheorem{theorem}{Theorem}[section]
\newtheorem{lemma}[theorem]{Lemma}
\newtheorem{corollary}[theorem]{Corollary}

\theoremstyle{definition}
\newtheorem{definition}{Definition}[section]
\newtheorem{example}{Example}[section]
\newtheorem{assumption}{Assumption}[section]

\theoremstyle{remark} 
\newtheorem{remark}{Remark}[section]

\newcommand{\bN}{\mathbb{N}}

\newcommand{\bP}{\mathbb{P}}
\newcommand{\bR}{\mathbb{R}}

\newcommand{\cA}{\mathcal{A}}
\newcommand{\cF}{\mathcal{F}}

\newcommand{\cN}{\mathcal{N}}
\newcommand{\cO}{\mathcal{O}}

\newcommand{\abs}[1]{{\left| #1 \right|}}
\newcommand{\ceil}[1]{{\left\lceil #1 \right\rceil}}

\newcommand{\1}[1]{{\ensuremath{\mathbf{1}_{\left\{ #1 \right\}}}}}

\newcommand{\pr}[1]{\left( #1 \right)}

\newcommand{\Prob}[1]{\mathbb{P}\mspace{-1mu}\left( #1 \right)}
\newcommand{\E}[1]{{\ensuremath{\mathbb{E}}\mspace{-2mu}\left[#1\right]}}
\newcommand{\Var}[1]{{\ensuremath{\mathrm{Var}\!\left( #1 \!\right)}}}

\newcommand{\toDist}{\xrightarrow[]{\rm{d}}}
\newcommand{\approxDist}{\overset{\rm{d}}{\approx}}

\newcommand{\DlX}{\Delta_\ell X} 

\newcommand{\defeq}{\mathrel{\vcenter{\baselineskip0.5ex \lineskiplimit0pt
                     \hbox{\scriptsize.}\hbox{\scriptsize.}}}%
                     =}

\newcommand{\Seb}[1]{\textcolor{black}{#1}}
\newcommand{\Hak}[1]{\textcolor{black}{#1}}

\begin{document}

\begin{frontmatter}
\title{Central limit theorems for multilevel Monte Carlo methods}
%
\author[Chalmers]{H{\aa}kon Hoel\corref{cor1}}
\ead{hhakon@chalmers.se, haakonah1@gmail.com}
\cortext[cor1]{Corresponding author}
\author[EPFL]{Sebastian Krumscheid}
\ead{sebastian.krumscheid@epfl.ch}
%
\address[Chalmers]{Department of Mathematical Sciences, Chalmers
  University of Technology and University of Gothenburg, SE-412 96
  Gothenburg, Sweden}
\address[EPFL]{Calcul Scientifique et Quantification de l'Incertitude
  (CSQI), Institute of Mathematics, {\'E}cole Polytechnique
  F{\'e}d{\'e}rale de Lausanne, CH-1015 Lausanne, Switzerland}
%



\begin{abstract}
  In this work, we show that uniform integrability is not a necessary
  condition for central limit theorems (CLT) to hold for normalized
  multilevel Monte Carlo (MLMC) estimators and we provide near optimal
  weaker conditions under which the CLT is achieved. In particular, if
  the variance decay rate dominates the computational cost rate (i.e.,
  $\beta > \gamma$), we prove that the CLT applies to the
  standard (variance minimizing) MLMC estimator.
  For other settings where the CLT may not apply to the
  standard MLMC estimator, we propose an alternative estimator, called
  the mass-shifted MLMC estimator, to which the CLT always applies.
  This comes at a small efficiency loss:
  the computational cost of achieving mean square approximation
  error $\cO(\epsilon^2)$ is at worst a factor $\cO(\log(1/\epsilon))$
  higher with the mass-shifted estimator than with the standard one.
  \end{abstract}


\end{frontmatter}


\section{Introduction}

The multilevel Monte Carlo (MLMC) method is a hierarchical sampling
method which in many settings improves the computational efficiency of
weak approximations by orders of magnitude.  The method was
independently introduced in the papers~\cite{MR1629093, MR2436856} for
the purpose of parametric integration and for approximations of
observables of stochastic differential equations, respectively.  MLMC
methods have since been applied with considerable success in a vast
range of stochastic problems, a collection of which can be found in
the overview~\cite{MR3349310}. In this work we present near optimal
conditions under which the normalized MLMC estimator converges in
distribution to a standard normal distribution. Our result has
applications in settings where the MLMC approximation error is
measured in terms of probability of failure~\eqref{eq:probFail}
%
%
rather than the classical mean square error.

\subsection{Main result}\label{sec:mainResult}

We consider the probability space $(\Omega, \cF, \bP)$ and let
$X \in L^2(\Omega)$ be a scalar random variable (r.~v.) for which we seek the
expectation $\E{X}$. Let
$\{X_\ell\}_{\ell=-1}^{\infty} \subset L^2(\Omega)$ be a sequence of r.~v.~satisfying the following:
\begin{assumption}\label{ass:mlmcRates}
  There exist rate constants $\alpha, \beta, \gamma >0$ with
  $\min(\beta, \gamma) \leq 2 \alpha$
  and a constant $c_\alpha > 0$ such that 
\begin{itemize}
 
 \item[(i)] $\abs{ \E{X - X_\ell} }   \leq c_\alpha 2^{-\alpha \ell}$ for all $\ell \in \bN_0 \defeq \bN \cup \{0\}$, 
 
 \item[(ii)] $V_0>0$ and $V_\ell \defeq \text{Var}(\DlX) = \cO_\ell(2^{-\beta \ell})$,
   
 \item[(iii)] $C_\ell \defeq \text{Cost}(\DlX) = \Theta_\ell\pr{2^{\gamma \ell}}$,
  
\end{itemize}
where $\DlX \defeq X_\ell-X_{\ell-1}$ with $X_{-1} \defeq 0$.
The notation $f(x_\ell) = \cO_\ell(y_\ell)$ means there exists a constant
$C >0$ such that $|f(x_\ell)| < C |y_\ell|$ for all $\ell \in \bN_0 := \bN \cup\{0\}$
and $f(x_\ell) = \Theta_\ell(y_\ell)$ means there exist
constants $C>c>0$ such that
$c|y_\ell| < |f(x_\ell)| < C |y_\ell|$ for all $\ell \in \bN_0$.
\end{assumption}

\begin{definition}[Variance minimizing
  MLMC estimator~\cite{MR3349310,MR1629093}]\label{def:mlmc}
  The MLMC estimator $\cA_{ML}\colon (0,\infty) \to L^2(\Omega)$
  applied to estimate the expectation of $X \in L^2(\Omega)$
  based on the collection of r.v.~$\{X_\ell\} \subset L^2(\Omega)$ satisfying
  Assumption~\ref{ass:mlmcRates} is defined by
  \[
    \cA_{ML}(\epsilon) = \sum_{\ell=0}^{L(\epsilon)} \sum_{i=1}^{M_\ell(\epsilon)} \frac{\DlX^i}{M_\ell(\epsilon)}\;.
  \]
  Here
  \[
    L^2(\Omega)\ni \Delta_\ell X^i = X_{\ell}^i-X_{\ell-1}^i, \quad \ell \in \bN_0,  \quad i \in \bN
  \]
  denotes a sequence of independent r.v.~and every subsequence
  $\{\Delta_\ell X^i\}_i$ consist of independent and identically
  distributed (i.i.d.) r.v., the number of levels is
  \begin{equation}\label{eq:L}
    L(\epsilon) \defeq \max\left(\ceil{\frac{\log_2(c_\alpha \epsilon^{-1}) }{\alpha}}, 1\right), \quad \epsilon >0,
  \end{equation}
  and the number of samples per level $\ell = 0,1, \ldots$ is
  \begin{equation}\label{eq:Ml}
    M_\ell(\epsilon) \defeq \max\left( \ceil{\epsilon^{-2} \sqrt{\frac{V_\ell}{C_\ell}} S_{L(\epsilon)}}, 1 \right), \quad \epsilon>0\;,
  \end{equation}
  with the monotonically increasing sequence $S_k$ defined as
    \begin{equation}\label{eq:SLDef}
      S_k \defeq \sum_{\ell =0}^k \sqrt{V_\ell C_\ell}, \quad k \in \bN_0\;.
    \end{equation}
  
  For any fixed and sufficiently large computational budget $c>0$,
  the sequence $\{M_\ell\}_{\ell =0}^L$ in~\eqref{eq:Ml} is the
  one in $\bN^{L}$ that minimizes $\Var{\cA_{ML}}$
  subject the constraint $\text{Cost}(\cA_{ML}) \le c$, cf.~\cite{MR2436856}.
  We will therefore refer to $\cA_{ML}$ as the variance minimizing
  MLMC estimator.
\end{definition}

\Seb{It is known that MLMC estimators can offer significant complexity
  (i.e., cost vs.~accuracy) benefits compared to classic Monte Carlo
  estimators \cite{MR3349310}. In fact, the variance minimizing estimator
  $\cA_{ML}(\epsilon)$ reduces the computational cost for achieving an
  approximation with mean square error of
  $\cO_\epsilon\bigl(\epsilon^2\bigr)$ from
  $\Theta_\epsilon\bigl(\epsilon^{-\left(2+\frac{\gamma}{\alpha}\right)}\bigr)$
  for the classic Monte Carlo method to
  $\Theta_\epsilon\bigl(\epsilon^{-2} S^2_{L(\epsilon)} \Hak{+ C_{L(\epsilon)}} \bigr)$, where
  \begin{equation*}
    S_{L(\epsilon)}= \begin{cases}
      \cO_\epsilon(1)& \text{if} \quad \beta>\gamma\;,\\
      \cO_\epsilon\bigl(\log\bigl(\epsilon^{-1}\bigr)\bigr)& \text{if} \quad \beta = \gamma\;,\\
      \cO_\epsilon\bigl(\epsilon^{-\frac{\gamma-\beta}{2\alpha}}\bigr)& \text{if} \quad \beta < \gamma\;,\\
    \end{cases}
  \end{equation*}
  and $C_{L(\epsilon)} = \Theta_\epsilon(\epsilon^{-\gamma/\alpha})$
  as functions of the rate triplet introduced in
  Assumption~\ref{ass:mlmcRates}.}

\Seb{In this work, we address the asymptotic normality of the MLMC
  estimator. For convenience, we}
will refer to
\[
  \frac{\cA_{ML}(\epsilon) - \E{X_{L(\epsilon)}}}{\sqrt{\Var{\cA_{ML}(\epsilon)}}}
\]
as the normalized estimator.
\Seb{When confusion is not possible,
  we will use the following shorthands,
  \[
    \cA_{ML} \defeq \cA_{ML}(\epsilon)\;, \quad M_\ell \defeq M_\ell(\epsilon)\;,
    \quad L \defeq L(\epsilon).
  \]
  The following conventions will be employed
  throughout this work:
  \[
    0\cdot (\pm \infty) = 0 \quad \text{and} \quad 0/0 = 0\;.
  \]
}

We are ready to state the main result of this work. 
\begin{theorem}[Main result]\label{thm:mainResult}
  Let $\cA_{ML}$ denote the variance minimizing MLMC estimator applied
  to estimate the expectation of $X \in L^2(\Omega)$ based on the
  collection of r.v.
  $\{X_\ell\} \subset L^2(\Omega)$ satisfying
  Assumption~\ref{ass:mlmcRates}. Additionally, if 
  \begin{enumerate}

  \item[(i)]  $\beta > \gamma$, impose no further assumptions, 
  \item[(ii)]  $\gamma \ge \beta$ and $\lim_{\ell \to \infty} S_\ell < \infty$, impose no further assumptions, 
    
  \item[(iii)] $\beta = \gamma$ and $\lim_{\ell \to \infty} S_\ell = \infty$,
    assume that
   \begin{equation}\label{eq:limCond}
 \hspace{-1cm} \lim_{\ell \to \infty} \1{V_\ell>0}\E{ \frac{ \abs{\DlX  - \E{\DlX}}^2}{ V_\ell}  \1{\frac{ \abs{\DlX  -
      \E{\DlX}}^2}{V_\ell} > \nu S_{\ell}^2 \exp\pr{(2\alpha - \gamma)\ell} }} = 0 \quad \forall \nu>0,
 \end{equation}
   
\item[(iv)] $\gamma>\beta$ and $\lim_{\ell \to \infty} S_\ell = \infty$,
  assume that $\beta< 2\alpha$,
  equality~\eqref{eq:limCond} holds and that there exists an
  $\upsilon \in [\beta,2\alpha)$ such that
  $\lim_{k \to \infty} S_k 2^{(\upsilon - \gamma)k/2} >1$.
    
\end{enumerate}
Then the normalized estimator satisfies the central limit
theorem (CLT), in the sense that
\begin{equation}\label{eq:cltMain}
  \frac{\cA_{ML} - \E{X_{L}}}{\sqrt{\Var{\cA_{ML}}}} \toDist \cN(0,1)
  \quad \text{as} \quad \epsilon \downarrow 0.
\end{equation}
  
\end{theorem}
The main result follows from Theorems~\ref{thm:betaLargest}
and~\ref{thm:gammaLargest}.  We note that Theorem~\ref{thm:mainResult} in
particular implies that the CLT always applies to the normalized
variance minimizing MLMC estimator when $\beta>\gamma$.

\begin{remark}
  The reason
  why we have not included the setting $\gamma>\beta$ and $\beta=
  2\alpha$ in Theorem~\ref{thm:mainResult} is that one cannot impose
  reasonable assumptions to exclude $M_L = \Theta_\epsilon(1)$ and
  $V_L/\Var{\cA_{ML}} = \Theta_\epsilon(1)$;
  cf.~Example~\ref{ex:fail1}. In such cases, a non-negligible
  contribution to the variance of the normalized estimator may
  derive from a finite number of samples on the finer
  levels $L, {L-1},\ldots$.
  For example, if
  $M_L = 1$ and $V_L/\Var{\cA_{ML}} \ge c>0$ for all $\epsilon>0$
  sufficiently small, then
  \[
  \frac{\cA_{ML} - \E{X_{L}}}{\sqrt{\Var{\cA_{ML}}}} =
  \sum_{\ell=0}^{L-1} \pr{\sum_{i=1}^{M_\ell} \frac{\DlX^i}{\sqrt{\Var{\cA_{ML}}} M_\ell}}
    +  \frac{\Delta_{L} X^1 \Seb{-\E{X_{L}}}}{\sqrt{\Var{\cA_{ML}}}}\;,
  \]
  and the CLT applies only if $\Delta_{\ell} X$ converges in
  distribution to a Gaussian as $\ell \to \infty$.  
\end{remark}

\subsection{Probability of failure}\label{sec:probFail}

Distributional properties of normalized sample estimators can be
useful for controlling the probability of (approximation) failure:
\begin{equation}\label{eq:probFail}
 \Prob{\abs{ \cA_{ML} - \E{X} } \ge 2\epsilon } \leq \delta\;.
\end{equation}
Here, $2\epsilon>0$ denotes the accuracy and $1-\delta >0$ the
confidence. To control the probability of failure, one
may dominate the total error from above by
the sum of a bias and a statistical error:
\begin{equation}\label{eq:splitError}
\Prob{\abs{ \cA_{ML} - \E{X} } \ge 2\epsilon } \leq
\Prob{\abs{ \E{X_L} - \E{X} } \ge \epsilon } +
\Prob{\abs{ \cA_{ML} - \E{X_L} } \ge \epsilon }\;.
\end{equation}
Assumption~\ref{ass:mlmcRates}(i) and the value of $L$ ensure that
the bias constraint is met
\[
\abs{ \E{X_L} - \E{X} } \le \epsilon.
\]
Supposing next that the CLT applies, the
key step in (approximately) controlling the statistical error
is the approximation
\[
\frac{\cA_{ML} - \E{X_L}}{\sqrt{\Var{\cA_{ML}}}} \approxDist \cN(0,1).
\]
The use of CLT in efficient algorithms for controlling the
probability of failure is a motivation for the goal of this
work: to describe as weak as possible conditions under which
the CLT applies to the standard MLMC estimator.

\begin{remark}
  Whenever $\beta\ge \gamma$ and $\alpha > \gamma/2$,
  one may reduce the bias of the variance minimizing
  MLMC estimator
  without affecting the asymptotic growth rate of the computational
  cost by replacing the rate parameter
  $\alpha$ by $\gamma/2$ in the formula for $L$ in~\eqref{eq:L} and updating
  the values for $\{M_\ell\}_{\ell=0}^L$ accordingly.
  This replacement leads to an asymptotically vanishing bias
  to standard deviation ratio,
  \[
  \lim_{\epsilon \downarrow 0} \frac{\E{X_L} - \E{X}}{\sqrt{\Var{\cA_{ML}}}} =
  \lim_{\epsilon \downarrow 0}  \epsilon^{2\alpha/\gamma-1} =0,
  \]
  and it relates to an uneven splitting of the accuracy between
  the bias and the statistical error constraints in~\eqref{eq:splitError}.
  That is,
  \[
  \begin{split}
  \Prob{\abs{ \cA_{ML} - \E{X} } \ge 2\epsilon } &\leq 
  \Prob{\abs{ \E{X_L} - \E{X} } \ge \theta(\epsilon)  \epsilon }\\
  &+\Prob{\abs{ \cA_{ML} - \E{X_L} } \ge (2-\theta(\epsilon))\epsilon }\;
  \end{split}
  \]
  for any monotonically increasing function
  $\theta:(0,\infty) \to (0,1]$ satisfying
  $\theta(\epsilon) \ge (\epsilon/c_\alpha)^{2\alpha/\gamma-1}$,
  cf.~\cite{MR3348197}.
  We leave as a remark that by
  straightforward extension of Theorem~\ref{thm:mainResult},
  the CLT also applies to the normalized variance minimizing
  MLMC estimator with $\theta$-splitting in settings where
  $\beta \ge \gamma$ and Theorem~\ref{thm:mainResult}'s
  assumptions hold.
\end{remark}

\subsection{The mass-shifted MLMC estimator}

In~\cite{glynnRhee2015,glynnZheng2016,glynnZheng2017}
Glynn et al.~show that for a collection of r.v. $\{X_\ell\}_{\ell=-1}^\infty$
satisfying Assumption~\ref{ass:mlmcRates} one can construct
the following unbiased coupled sampling method for the limit r.v.~$X$:
  \[
    Z = \sum_{\ell=0}^{\infty} \frac{\Delta_\ell X \, \, \1{N \ge \ell} }{\Prob{N \ge \ell}}.
  \]
  Here, the r.v.~$N: \Omega \to \bN_0$
  is independent of $\{\Delta_\ell X\}_{\ell=-1}^\infty$
  and $\Prob{N \ge \ell} > 0$ for all $\ell \ge 0$.
  Provided $N$ is chosen such that
  $\E{|Z|} < \infty$, the strong law of large numbers
  yields that
  \[
    \overline Z_M = \frac{1}{M} \sum_{i=0}^{M}
    Z^{i}  \overset{\rm{a.s.}}{\to} \E{X} \quad
      \text{as} \quad M \to \infty,
  \]
  where $Z^{1},Z^{2}, \ldots$ is an i.i.d.~sequence with $Z^{i} \stackrel{d}{=} Z$.
  Although $\overline Z_M$ clearly is not an MLMC estimator
  of the kind studied in this paper,
  one may view it,
  when the number of samples $M$ is large,
  as a randomized MLMC estimator where both
  $L$ and $M_\ell \approx M\times \Prob{N\ge\ell} $ for all $\ell\ge0$
  are random non-negative numbers, cf.~\cite{glynnZheng2016}.
  By carefully choosing the
  distribution of $N$ such that $\Var{Z^{i}}<\infty$
  and exploiting that $\overline Z_M$
  is the sum of i.i.d.~random variables, Glynn et al.~prove that the CLT
  applies to
  $(\overline Z_M - \E{X})/\sqrt{\Var{ \overline Z_M}}$ in settings where
  $\beta \ge \gamma$.

  Concerning the efficiency of the method, it can be shown that the distribution $N$
  that minimizes the quantity
  $\Var{\overline Z_M} \times \text{Cost}(\overline Z_M)$,
  satisfies
  \begin{equation}\label{eq:pLim}
   \Prob{N\ge \ell } = \Theta_\ell( \sqrt{V_\ell/C_\ell})
  \end{equation}
  (supposing, unlike our approach, that $V_\ell >0$ for all $\ell$).
  When $\beta >\gamma$, any distribution $N$ satisfying~\eqref{eq:pLim}
  induces a distribution $Z$ that has bounded variance, and
  consequently, the CLT applies. When $\beta =\gamma$, however,
  it turns out that $\Var{Z}=\infty$ for any $N$ satisfying~\eqref{eq:pLim},
  so that in order to obtain the CLT
  one needs to consider distributions $N$ whose mass is
  shifted slightly from the efficiency optimizing~\eqref{eq:pLim}
  to the tail:
    \[
      \Prob{N\ge \ell } = \Theta_\ell( (\ell+1) \log(\ell+2)^{1+\xi}
      \sqrt{V_\ell/C_\ell}), \qquad \xi >0.
   \]
   This shift leads to an estimator $\overline Z_M$ with
   approximation error
   \mbox{$\E{ \pr{\overline Z_M - \E{Z}}^2  }= \cO_{\epsilon}(\epsilon^2)$} obtained at the (random) computational
  cost $\cO_\epsilon(\epsilon^{-2} \log(1/\epsilon)^2
  \log(\log(1/\epsilon))^{1+\xi})$.  In comparison, for the settings
  covered by Theorem~\ref{thm:mainResult} when $\beta =\gamma$,
  the variance minimizing estimator $\cA_{ML}(\epsilon)$ achieves
  the MSE $\cO_\epsilon(\epsilon^2)$ at the slightly lower (and non-random)
  computational cost $\Theta_\epsilon(\epsilon^{-2} S_L^2) =
  \cO_{\epsilon}(\epsilon^{-2} \log(1/\epsilon)^2)$.

    Taking inspiration of from Glynn et al.'s mass-shifting approach, we propose the following 
    relative shift of ``sample mass'' from the lower levels of
    the variance minimizing estimator's optimal $\{M_\ell\}_{\ell=0}^L$ to the higher levels:
    \begin{equation}\label{eq:mlTilde}
      \widetilde M_\ell := \max\pr{ \ceil{
        \epsilon^{-2} (S_\ell+1) \log(S_{\ell}+1)^{1+\xi} \sqrt{\frac{V_\ell}{C_\ell}}   \widetilde S_L }, \, 1},
    \end{equation}
    where 
    \[
      \widetilde S_L := \sum_{\ell=0}^L  \frac{\sqrt{V_\ell C_\ell}}{(S_\ell+1) \log(S_{\ell}+1)^{1+\xi}},  \quad \xi >0,
    \]
    and the resulting estimator
    \begin{equation}\label{eq:massShiftedMLMC}
      \widetilde \cA_{ML} = \sum_{\ell=0}^{L} 
\sum_{i=1}^{\widetilde M_\ell} \frac{\DlX^i}{ \widetilde M_\ell}.
    \end{equation}
    We will refer to $\widetilde \cA_{ML}$ as the mass-shifted MLMC estimator.
    The CLT applies in all relevant settings for the normalized version of
    this estimator: 
    \begin{theorem}[CLT for mass-shifted MLMC]\label{thm:mainResult2}
      For any $\xi >0$, let $\widetilde \cA_{ML}$ denote the resulting mass-shifted 
      MLMC estimator applied to
      estimate the expectation of $X \in L^2(\Omega)$ based on the
      collection of r.v.
      $\{X_\ell\} \subset L^2(\Omega)$ satisfying
      Assumption~\ref{ass:mlmcRates}.
      Then the normalized mass-shifted MLMC estimator satisfies 
      \begin{equation}\label{eq:cltMain2}
        \frac{\widetilde \cA_{ML} - \E{X_{L}}}{\sqrt{\Var{\widetilde \cA_{ML}}}} \toDist \cN(0,1) \quad \text{as} \quad \epsilon \downarrow 0
      \end{equation}
      and the approximation error $\E{\pr{\widetilde
          \cA_{ML} - \E{X}}^2} = \cO(\epsilon^2)$ is achieved at the
      computational cost
  \[
  \Theta_{\epsilon}(\epsilon^{-2} (S_L+1)^2 \log(S_L+1)^{1+\xi} )
  = \begin{cases} 
    \cO_\epsilon(\epsilon^{-2}), & \beta > \gamma\\
    \cO_\epsilon(\epsilon^{-2} \log(1/\epsilon)^{2}
    \log(\log(1/\epsilon))^{1+\xi}), & \beta = \gamma\\
    \cO_\epsilon(\epsilon^{-2 - \frac{\gamma-\beta}{\alpha}} 
    \log(1/\epsilon)^{1+\xi}),
& \gamma>\beta.
\end{cases}
  \]
\end{theorem}
The proof of Theorem~\ref{thm:mainResult2} is given in Section~\ref{sec:massShiftProof}.

\subsection{Literature review}

In addition to the above mentioned contributions by Glynn et al.,
the CLT has been proved for MLMC methods through
assuming (or verifying for the particular sequence of r.v.~considered)
either a Lyapunov condition~\cite{hoel2014implementation}, or uniform
integrability~\cite{MR3297771,MR3449315,Giorgi}, or a weaker higher
moment decay rate~\cite{MR3348197} for the sequence
$\{\1{V_\ell>0}|\DlX-\E{\DlX}|^2/V_\ell\}_{\ell\in \bN_0}$.
To show that this work
extends the existing literature, we now provide an explicit example
that is covered by Theorem~\ref{thm:mainResult} but  
where uniform integrability does not hold.

\begin{example}
  Consider the stochastic differential equation
  \begin{equation}\label{eq:SDE}
  dY = a(Y) \,dt + b(Y) \, dW(t) \qquad t \in [0,T]  
  \end{equation}
  with final time $T>0$, initial condition $Y(0) \in \bR$,
  and coefficients $a,b: \bR \to \bR$ whose partial
  derivatives of all orders are continuous and uniformly bounded.
  For a given strike $K \in \bR$,
  we seek to approximate the expectation of the (non-discounted) digital option payoff
  $X = \1{Y(T)\ge K}$. Let $X_\ell = \1{Y_\ell(T)\ge K}$ denote the
  $\ell$-th resolution approximation of $X$ where $Y_\ell(T)$ 
  denotes the order 1.5 strong Ito-Taylor scheme~\cite[Ch.~10.4]{kloedenPlaten}
  numerical solution 
  using a uniform timestep $h_\ell = 2^{-\ell}T$.
  In order to minimize the variance,
  coupled realizations $Y_\ell(\cdot,\omega)$ and $Y_{\ell-1}(\cdot,\omega)$
  use the same Wiener path sampled at different resolutions.
  Furthermore, the scheme's fine resolution integral increments of the form
  \[
    \Delta z_n^\ell = \int_{n h_\ell}^{(n+1)h_\ell} W(s)-W(nh_\ell) \, dt
    \stackrel{\rm d}{=} \frac{\Delta W^\ell_n h_\ell }{2} + \frac{h_\ell^{3/2}}{\sqrt{12}} \chi_n,
  \]
  where $\chi_n \sim N(0,1)$ and $\Delta W^\ell_n = W((n+1) h_\ell) - W(nh_\ell)$
  are independent, are coupled to overlapping coarse ones as follows:
  \[
    \begin{split}
    \Delta z_{n}^{\ell-1} &= \Delta z_{2n}^\ell
    + \int_{(2n+1) h_\ell}^{2(n+1)h_\ell} W(s)- W(2n h_{\ell}) \, dt\\
    &= \Delta z_{2n}^\ell + h_\ell \Delta W^\ell_{2n} + \int_{(2n+1) h_\ell}^{2(n+1)h_\ell} W(s)- W((2n+1) h_{\ell}) \, dt\\
    &= \Delta z_{2n}^\ell + h_\ell \Delta W^\ell_{2n} + \Delta z_{2n+1}^\ell.
    \end{split}
  \]
  (That is, first generate $(\Delta z_{2n}^\ell, \Delta z_{2n+1}^\ell)(\omega)$,
  $\Delta W^\ell_{2n}(\omega)$ and $\Delta W^\ell_{2n+1}(\omega)$, then
  compute the overlapping coupled coarse increment $\Delta z_{n}^{\ell-1}(\omega)$
  by the above formula.)
  Assuming that the diffusion coefficient is strictly positive
  and $b'|_{D}\neq 0$ in an open domain $D\subset\bR$ containing $Y(0)$ and $K$,
  \begin{equation}\label{eq:probDomain}
  \Prob{ \abs{Y_\ell(T)-K} \le h_\ell^{3/2}} = \cO_\ell(h_\ell^{3/2})
  \end{equation}
  and
  \begin{equation}\label{eq:limsupCond}
   \limsup_{\ell \to \infty} \; \text{ess} \sup_{\omega \in \Omega} |\DlX(\omega)-\E{\DlX}|^2 =1.
  \end{equation}
  By the order 1.5 strong order scheme,
  $Y_\ell(T) - Y_{\ell-1}(T) = \cO_\ell(h_\ell^{3/2})$, which together
  with~\eqref{eq:probDomain} imply that
  $V_\ell= \Var{\Delta_\ell X} = \cO_\ell(h_\ell^{3/2})$.
  Lastly, since $\text{Cost}(Y_\ell) = \Theta_\ell(1/h_\ell)$, the rate triplet for $\{X_\ell\}$ becomes $\alpha =1, \beta=3/2$ and
  $\gamma=1$.

  Note further that the sequence
  $\{\1{V_\ell>0}|\DlX-\E{\DlX}|^2/V_\ell\}_{\ell\in \bN_0}$ is not
  uniformly integrable since by~\eqref{eq:limsupCond},
  \[
    \limsup_{\ell \to \infty} \1{V_\ell>0} \frac{ \text{ess} \sup_{\omega \in \Omega} |\DlX(\omega)-\E{\DlX}|^2}{V_\ell} 2^{-\beta \ell}  >0, 
  \]
 which implies that  
  \[
    \limsup_{\ell \to \infty} \1{V_\ell>0} \E{ \frac{|\DlX-\E{\DlX}|^2 }{V_\ell} \1{\frac{\abs{\DlX - \E{\DlX}}^2 }{V_\ell} > x} }
    =1, \quad \text{for any } x>0.
  \]
  Regardless of uniform integrability, however, the CLT applies
  according to Theorem~\ref{thm:mainResult} in 
  the current setting of $\beta>\gamma$.
\end{example}


\subsection*{Applications of MLMC}

We conclude this section with a brief survey on the relationship
between the rate parameters $\beta$ and $\gamma$ from
Assumption~\ref{ass:mlmcRates} for a couple of problems
which have been frequently studied.
  
  As a first example, consider
  the quantity of interest (QoI) $X=\varphi(Y) \in\mathbb{R}$
  with $Y:[0,T]\times \Omega \to \bR$ denoting the solution of an SDE
  of the form~\eqref{eq:SDE}. For an approximation sequence
  $X_\ell = \phi(Y_\ell)$, where $Y_\ell$ is generated by a
  numerical method with uniform timestep $h_\ell = 2^{-\ell}T$,
  one often obtains $C_\ell = \text{Cost}(X_\ell) = \cO(h_\ell^{-1})$,
  yielding $\gamma=1$ (this applies for instance to the
  Euler--Maruyama and the Milstein schemes).
  The variance decay rate $\beta$ is
  typically more sensitive, as it tends to depend on both the 
  strong order of convergence of the numerical method and the
  regularity of the functional $\varphi$. If
  the SDE coefficients and the QoI are all sufficiently
  regular, then $\beta=1$ for the
  Euler--Maruyama scheme and $\beta =2$ for the
  Milstein scheme, but low-regularity QoIs
  often lead to lower-valued $\beta$. For instance, for digital and
  barrier options, $\beta=1/2$ for Euler--Maruyama
  and $\beta=1$ for Milstein (provided no further smoothing
  is applied), cf.~\cite[Sec.~5]{MR3349310}.
  Similar reductions in the variance decay rate may occur
  if the SDE coefficients have low regularity or
  if its driving path has lower regularity than
  a Wiener process, 
  cf.~\cite{MR3501366, hoel2016}.

  As a second example, let the quantity of interest be
  $X = \varphi(u) \in\mathbb{R}$, where
  $u(\omega,\cdot)\colon D\to\mathbb{R}$ denotes
  the solution of the linear elliptic partial
  differential equation (PDE)
  \begin{equation*}
    -\operatorname{div}\left(a(\omega,x)\nabla u(\omega,x)\right) = f(\omega,x)\;,\quad\text{in }D\subset\mathbb{R}^d, \quad \omega \in \Omega\;,
  \end{equation*}
  with random coefficient functions
  $a(\omega,\cdot)\colon D\to\mathbb{R}$ and
  $f(\omega,\cdot)\colon D\to\mathbb{R}$, equipped with suitable
  boundary conditions. Similarly to the SDE problem above, the
  lower the regularity of the random coefficients and/or
  the functional $\varphi$, the lower the variance decay rate $\beta$
  becomes, cf.~\cite{Teckentrup2013}.
  Moreover, the computational cost rate $\gamma$
  is typically proportional to the dimension $d$ of the
  spatial domain $D$. 

  Finally, let us mention that MLMC has been successfully
  applied to a wide range applications, such as seismic wave
  propagation~\cite{BallesioEtAl}, stochastic reaction
  networks~\cite{anderson, moraes}, stochastic partial differential
  equations~\cite{barthLang2, mishra},
  optimal experimental design~\cite{BeckEtAl},
  Markov chain Monte Carlo simulation~\cite{scheichlMcmc, hoang},
  Bayesian inversion and filtering methods~\cite{lawJasra, lawHoel, ullmannBayes, reich},
  and rare event estimation/importance sampling~\cite{ullmannRare, kebaierImpSampling},
  to name but a few.  As a
  consequence of these applications' diverse nature, a wide variety of
  different rate triplet scenarios is commonly relevant in
  practice.

\section{Theory}\label{sec:cltTheory}

In this section we derive weak assumptions under which the normalized
MLMC estimator $(\cA_{ML} -\E{X_L})/\sqrt{\Var{\cA_{ML}}}$ converges
in distribution to a standard normal as $\epsilon \to 0$.  The main
tool used for verifying the CLT will be the Lindeberg condition, which
in its classical formulation is an integrability condition for
triangular arrays of independent random variables (r.v.) $Y_{nm}$,
with $n \in \bN$ and $1\le m \le k_n$; cf.~\cite{MR1609153}.  However,
in the multilevel setting it is more convenient to work with
generalized triangular arrays of independent r.v.\ of the form
$Y_{\epsilon m}$, which for a fixed $\epsilon>0$ take possible
non-zero elements within the set of indices $ 1\le m \le n(\epsilon)$,
where $n\colon (0,\infty) \to \bN$ is a strictly decreasing function
of $\epsilon>0$ with $\lim_{\epsilon \downarrow 0} n(\epsilon) = \infty$.

The following theorem is a trivial extension of~\cite{klenke} from
triangular arrays to generalized triangular arrays.
\begin{theorem}[Lindeberg-Feller Theorem]\label{thm:LindebergFeller}
  For every $\epsilon >0$, let $\{Y_{\epsilon m}\}$,
  $ 1\leq m \leq n(\epsilon)$ with $n\colon (0,\infty) \to \bN$ and
  $\lim_{\epsilon \downarrow 0} n(\epsilon) = \infty$ be a generalized
  triangular array of independent random variables that are centered
  and normalized, so that
  \begin{equation}\label{eq:normCenter}
    \E{Y_{\epsilon m}} = 0\quad\text{and}\quad   \sum_{m=1}^{n(\epsilon)} \E{Y_{\epsilon m}^2} = 1\;,
  \end{equation}
  respectively. Then, the Lindeberg condition: 
  \begin{equation}\label{eq:LFCondition1}
    \lim_{\epsilon \downarrow 0} 
    \sum_{m=1}^{n(\epsilon)} \E{Y_{\epsilon m}^2 \1{|Y_{\epsilon m}|>\nu }} = 0\quad\forall\,\nu>0\;,
  \end{equation}
  holds, if and only if
  \begin{equation}\label{eq:extendedClt}
    \sum_{m=1}^{n(\epsilon)}
    Y_{\epsilon m}  \toDist  \cN(0,1) \text{ as } \epsilon \downarrow 0 \quad \text{and} \quad 
    \lim_{\epsilon \downarrow 0} \max_{m \in \{1,2,\ldots,n(\epsilon)\}}
    \E{Y_{\epsilon m}^2} = 0\;.
  \end{equation}
\end{theorem}

We will refer to~\eqref{eq:extendedClt} as the \emph{extended CLT
  condition}.  By defining
\begin{equation}\label{eq:nDef}
n(\epsilon) \defeq \sum_{\ell=0}^L M_\ell,
\end{equation}
and
\begin{equation}
Y_{\epsilon m} \defeq \begin{cases} 
\frac{ \Delta_0 X^m - \E{\Delta_0X}}{ \sqrt{\Var{\cA_{ML}}} M_0} &  m \le M_0\\
\frac{\Delta_1 X^m - \E{\Delta_1 X}}{\sqrt{\Var{\cA_{ML}} } M_1} &  M_0 < m \le M_0 + M_{1}\\
\qquad\vdots\\
\frac{\Delta_L X^m - \E{\Delta_L X}}{\sqrt{\Var{\cA_{ML}}  } M_L} &
n(\epsilon)-M_L < m \le n(\epsilon),
\end{cases}
\end{equation}
the normalized variance minimizing MLMC estimator
can be represented by generalized triangular arrays as follows:
\begin{equation}\label{eq:scaledML}
  \frac{\cA_{ML} - \E{X_L}}{\sqrt{\Var{\cA_{ML}}}} = \sum_{m=1}^{n(\epsilon)} Y_{\epsilon m}\;.
\end{equation}
  \Hak{We note that the telescoping property
    $\E{X_L} = \sum_{\ell=0}^L \E{\DlX}$
    was used to obtain~\eqref{eq:scaledML}. Moreover, the representation~\eqref{eq:scaledML}
    and the below corollary trivially extends to any normalized MLMC estimator.
    }

\begin{corollary}\label{cor:lindebergGeneral}
  Let $\cA_{ML}$ denote the variance minimizing MLMC estimator applied to estimate
    the expectation of $X \in L^2(\Omega)$ based on the collection of
    r.v.
  $\{X_\ell\} \subset L^2(\Omega)$ satisfying
  Assumption~\ref{ass:mlmcRates}.  Suppose that $\Var{\cA_{ML}}>0$ for
  any $\epsilon >0$.  Then the normalized
  estimator~\eqref{eq:scaledML} satisfies the extended CLT
  condition~\eqref{eq:extendedClt}, if and only if for any $\nu>0$,
\begin{equation}\label{eq:lindebergGeneral}
\lim_{\epsilon \downarrow 0} \sum_{\ell=0}^L
\frac{ V_\ell}{ \Var{\cA_{ML}} M_\ell } \E{ \frac{ \abs{\DlX  - \E{\DlX}}^2}{V_\ell}  \1{\frac{ \abs{\DlX  -
      \E{\DlX}}^2}{V_\ell} > \frac{\Var{\cA_{ML}} M_\ell^2}{V_\ell} \nu}} =0.
\end{equation}

\end{corollary}
\begin{proof}
  For all $\epsilon>0$, the triangular array
  representation~\eqref{eq:scaledML} of the MLMC estimator obviously
  satisfies the centering and normalization
  conditions~\eqref{eq:normCenter}, and its elements are centered and
  mutually independent. By Theorem~\ref{thm:LindebergFeller}, the
  extended CLT condition thus holds if and only if Lindeberg's
  condition~\eqref{eq:LFCondition1} holds. For any $\nu >0$, here
  Lindeberg's condition takes the form: 
  \[
    \begin{split}
      &\lim_{\epsilon \to 0} \sum_{m=1}^{n(\epsilon)} \E{ Y_{\epsilon m}^2
        \1{\lvert Y_{\epsilon m}\rvert > \nu}}  \\
      &= \lim_{\epsilon \to 0}  \sum_{\ell=0}^L \sum_{i=1}^{M_\ell}
       \E{ \frac{ \abs{\DlX^i  -
            \E{\DlX}}^2}{M_\ell^2 \Var{\cA_{ML}}}  \1{\frac{ \abs{\DlX^i  -
              \E{\DlX}}^2}{\Var{\cA_{ML}} M_\ell^2} > \nu^2}}\\
      & = \lim_{\epsilon \downarrow 0} \sum_{\ell=0}^L \frac{V_\ell}{M_\ell\Var{\cA_{ML}}} \E{ \frac{ \abs{\DlX  -
            \E{\DlX}}^2}{V_\ell}  \1{\frac{ \abs{\DlX  -
              \E{\DlX}}^2}{V_\ell} > \frac{\Var{\cA_{ML}} M_\ell^2}{V_\ell}
          \nu^2}}.
    \end{split}
  \]
\end{proof}
Assumption~\ref{ass:mlmcRates} does not provide any lower bound on the
decay rate of the variance sequence $\{V_\ell\}$, and therefore it
alone is not sufficiently strong to ensure that Lindeberg's
condition~\eqref{eq:lindebergGeneral} holds in general. The problem is
that without any lower bound on $V_\ell$, there are asymptotic
settings where a non-negligible contribution to the variance of the
variance minimizing MLMC estimator derives from a finite number of samples.
\begin{example}\label{ex:fail1}
  Consider the setting where $\beta \le 2\alpha < \gamma$, for some 
  constants $c_2>c_1>0$,
\[
 c_1 2^{-2\alpha \ell}\le V_\ell \le c_2 2^{-\beta\ell} \qquad \forall \ell \in \bN_0,   
\]
and for an infinite subsequence $\{k_i\} \subset \bN_0$,
\[
V_{k_i} = \Theta_{i}(2^{-2\alpha k_i}) \quad \text{and} \quad S_{k_i} = \Theta_{i}(2^{(\gamma-2\alpha)k_i/2})
\quad \forall i \in \bN_0.
\]
Then equation~\eqref{eq:Ml} implies there exists $c,C,\tilde{c},\hat{c} \in \bR_+$ such that
for all $y \in \{\epsilon>0 \mid L(\epsilon) \in \{k_i\} \}$,
\[
1 \le M_{L(y)} < C,
\]
and
\[
\hat{c} \le \max\pr{\frac{V_{L(y)}}{M_{L(y)} \Var{\cA_{ML}(y)}}, \;
  \frac{M_{L(y)}^2 \Var{\cA_{ML}(y)}}{V_{L(y)}}} \le \tilde{c}.
\]
Hence, for any $\nu<(2\tilde{c})^{-1}$,
\[
\begin{split}
& \limsup_{\epsilon \downarrow 0} \sum_{\ell=0}^L
\frac{V_\ell}{ M_\ell \Var{\cA_{ML}} } \E{ \frac{ \abs{\DlX  -
      \E{\DlX}}^2}{V_\ell}  \1{\frac{ \abs{\DlX  -
      \E{\DlX}}^2}{V_\ell} > \frac{\Var{\cA_{ML}} M_\ell^2}{V_\ell}
  \nu}} \\
& \ge
\limsup_{\epsilon \downarrow 0} 
\frac{V_L}{ M_L \Var{\cA_{ML}} } \E{ \frac{ \abs{\Delta_L X  -
      \E{\Delta_LX}}^2}{V_L}  \1{\frac{ \abs{\Delta_L X  -
      \E{\Delta_L X}}^2}{V_L} > \frac{\Var{\cA_{ML}} M_L^2}{V_L}
    \nu}}\\
& \ge
\limsup_{i \to \infty} 
\hat{c} \, \E{ \frac{ \abs{\Delta_{k_i} X  -
      \E{\Delta_{k_i} X}}^2}{V_{k_i}}  \1{\frac{ \abs{\Delta_{k_i} X  -
      \E{\Delta_{k_i} X}}^2}{V_{k_i}} > \frac{1}{2}}} \ge \frac{\hat c}{2} > 0.
\end{split}
\]

\end{example}

Example~\ref{ex:fail1} illustrates that Assumption~\ref{ass:mlmcRates}
is not sufficiently strong to ensure condition~\eqref{eq:lindebergGeneral}
when $\gamma >\beta$.
We therefore impose the following additional variance decay assumption,
which can be viewed
as an implicit weak lower bound on the sequence $\{V_\ell\}$.

\begin{assumption}\label{ass:mlmcRates2}
  If Assumption~\ref{ass:mlmcRates} holds for a collection of
  r.v.~$\{X_\ell\} \subset L^2(\Omega)$ with limit $X \in L^2(\Omega)$
  in the setting $\gamma > \beta$ and $\lim_{\ell \to \infty} S_\ell = \infty$,
  then assume additionally that $\beta <2\alpha$ and that there exists
  an $\upsilon \in [\beta, 2\alpha)$ such that
    \[
      \liminf_{\ell \to \infty} S_\ell 2^{(\upsilon -\gamma)\ell/2} >1.
    \]
\end{assumption}

\begin{lemma}\label{lem:varEpsRelation}
  Let $\cA_{ML}$ denote the variance minimizing MLMC estimator applied to estimate
    the expectation of $X \in L^2(\Omega)$ based on the collection of
    r.v.
    $\{X_\ell\} \subset L^2(\Omega)$ satisfying
  Assumptions~\ref{ass:mlmcRates} and~\ref{ass:mlmcRates2}. Then
\begin{equation}\label{eq:condition1}
\lim_{\epsilon \downarrow 0} \frac{\Var{\cA_{ML}}}{\epsilon^2} =1\;.
\end{equation}
\end{lemma}
\begin{proof}
  For any $\epsilon >0$, it follows from equation~\eqref{eq:Ml} that
\[
\begin{split}
\frac{\Var{\cA_{ML}}}{\epsilon^2}  =
\sum_{\ell=0}^L \frac{V_\ell }{\epsilon^2 M_\ell} 
\leq \sum_{\ell=0}^L \frac{\sqrt{V_\ell C_\ell} }{S_L}= 1\;,
\end{split}
\]
and by the mean value theorem there exists a
constant $C>0$ such that
\begin{equation}\label{eq:condition1LowerBound}
\begin{split}
\sum_{\ell=0}^L \frac{V_\ell }{\epsilon^2 M_\ell} &\ge
\sum_{\ell=0}^L \1{V_\ell >0} \frac{V_\ell }{ \sqrt{\frac{V_\ell}{C_\ell}}S_L  + \epsilon^2}\\
&\ge 1- \sum_{\ell=0}^L  \1{V_\ell >0}\frac{V_\ell \epsilon^2}{\frac{V_\ell}{C_\ell} S_L^2 }\\
& \ge 1 - \epsilon^2 \frac{\sum_{\ell=0}^L  C_\ell}{S_L^2 }\\
&\ge 1 - C \epsilon^2\frac{ 2^{\gamma L}}{S_L^2}\;.
\end{split}
\end{equation}
To complete the proof, it remains to verify that
\begin{equation}\label{eq:condition1Closure}
\lim_{\epsilon \downarrow 0} \frac{\epsilon^2 2^{\gamma L}}{S_L^2} =0\;.
\end{equation}

We separate the proof into three cases:

{\bf (i):} If $\beta < \gamma$ and $\lim_{\ell \to \infty} S_\ell = \infty$,
then Assumption~\ref{ass:mlmcRates2} implies that
\[
\frac{\epsilon^2 2^{\gamma L}}{S_L^2} = \cO(\epsilon^{2-\upsilon/\alpha})\;,
\]
and since $\upsilon < 2 \alpha$, the claim follows. 

{\bf (ii):} If $\beta = \gamma$ and $\lim_{\ell \to \infty} S_\ell = \infty$,
then $\gamma \le 2\alpha$, cf.~Assumption~\ref{ass:mlmcRates}, implies
that $\epsilon^2 2^{\gamma L} = \cO_\epsilon(1)$ and the claim follows.

{\bf (iii):} If $\lim_{\ell \to \infty} S_\ell =:S < \infty$,
then there exists a $k>1$ such that $\gamma/k < 2\alpha$ and a
$C>0$ such that 
\[
\begin{split}
  \frac{\Var{\cA_{ML}}}{\epsilon^2}  &\ge
  \sum_{\ell=0}^{\ceil{L/k}} \1{V_\ell >0} \frac{V_\ell }{
    \sqrt{\frac{V_\ell}{C_\ell}} S  + \epsilon^2} \ge
  \frac{S_{\ceil{L/k}}}{S} - C \epsilon^2\frac{ 2^{\gamma L/k}}{S^2},
\end{split}
\]
The claim follows from $\lim_{\epsilon \downarrow 0} \epsilon^2 2^{\gamma L/k} = 0$
and $\lim_{\epsilon \downarrow 0} S_{\ceil{L/k}} = S$.

Case {\bf (iii)} covers all
settings $\gamma \ge \beta$ which 
are not covered by either {\bf (i) } or {\bf (ii)}.
Furthermore, since $S_\ell  = \cO_\ell(2^{(\gamma- \beta) \ell/2})$,
it is clear that {\bf (iii)} also covers 
all settings with $\beta > \gamma$.
This shows that cases {\bf (i)--(iii)} cover
all settings that are valid under
Assumptions~\ref{ass:mlmcRates}
and~\ref{ass:mlmcRates2}.


\end{proof}
Lemma~\ref{lem:varEpsRelation} implies that we can reformulate
Lindeberg's condition for the MLMC estimator as follows:

\begin{corollary}\label{cor:extCLT}
  Let $\cA_{ML}$ denote the variance minimizing MLMC estimator applied to estimate the
    expectation of $X \in L^2(\Omega)$ based on the collection of
    r.v.\ $\{X_\ell\} \subset L^2(\Omega)$ satisfying
    Assumptions~\ref{ass:mlmcRates} and~\ref{ass:mlmcRates2}.
    Then the normalized MLMC estimator satisfies the extended CLT
  condition~\eqref{eq:extendedClt}, if and only if for any $\nu >0$,
\begin{equation}\label{eq:lindebergCond2}
  \lim_{\epsilon \downarrow 0} \sum_{\ell=0}^L \frac{\sqrt{V_\ell C_\ell}}{S_L} \1{V_\ell >0} \E{\frac{\abs{\DlX - \E{\DlX}}^2}{V_\ell}
    \1{\frac{\abs{\DlX - \E{\DlX}}^2 }{V_\ell} > \frac{\epsilon^2 M_\ell^2}{V_\ell} \nu}} =0\;.
\end{equation}
\end{corollary}
\begin{proof}
  From the proof of Lemma~\ref{lem:varEpsRelation} it follows that
  there exists an $\bar \epsilon>0$ such that
\[
  \frac{1}{2} \le \frac {\Var{\cA_{ML}}}{\epsilon^2} \le 1\;, \quad
  \forall \epsilon \in (0, \bar \epsilon)\;.
\]
Consequently, for any $\epsilon \in(0, \bar \epsilon)$ and any $\nu>0$
we have that
\[
\begin{split}
& \sum_{\ell=0}^L \E{ \frac{ \abs{\DlX  -
      \E{\DlX}}^2}{\Var{\cA_{ML}}  M_\ell}  \1{\frac{ \abs{\DlX  -
      \E{\DlX}}^2}{V_\ell} > \frac{\Var{\cA_{ML}} M_\ell^2}{V_\ell} \nu}}\\
& \ge  \sum_{\ell=0}^L \E{ \frac{ \abs{\DlX  -\E{\DlX}}^2}{\epsilon^2 M_\ell}  \1{\frac{ \abs{\DlX  -
      \E{\DlX}}^2}{V_\ell} > \frac{\epsilon^2 M_\ell^2}{V_\ell} \nu}}\;,
\end{split}
\]
as well as
\[
\begin{split}
& \sum_{\ell=0}^L \frac{1}{\Var{\cA_{ML}}} \E{ \frac{ \abs{\DlX  -
      \E{\DlX}}^2}{M_\ell}  \1{\frac{ \abs{\DlX  -
      \E{\DlX}}^2}{V_\ell} > \frac{\Var{\cA_{ML}} M_\ell^2}{V_\ell} \nu}}\\
& \le  2\sum_{\ell=0}^L  \E{ \frac{ \abs{\DlX  -
      \E{\DlX}}^2}{\epsilon^2 M_\ell}  \1{\frac{ \abs{\DlX  -
      \E{\DlX}}^2}{V_\ell} > \frac{\epsilon^2 M_\ell^2}{2V_\ell} \nu}}\;.
\end{split}
\]
These upper and lower bounds imply that that Lindeberg's
condition~\eqref{eq:lindebergGeneral} is equivalent to the following
condition: for any $\nu>0$ it holds that
\[
\lim_{\epsilon \downarrow 0} \sum_{\ell=0}^L \E{ \frac{ \abs{\DlX  -
      \E{\DlX}}^2}{\epsilon^2 M_\ell}  \1{\frac{ \abs{\DlX  -
      \E{\DlX}}^2}{V_\ell} > \frac{\epsilon^2 M_\ell^2}{V_\ell} \nu}}
= 0\;.
\]

Following similar steps as those leading to
inequality~\eqref{eq:condition1LowerBound}, we further note that for
sufficiently small $\epsilon>0$,
\begin{equation}\label{eq:condition2_1}
\begin{split}
&\sum_{\ell=0}^L \frac{1}{\epsilon^2 M_\ell}\E{\abs{\DlX - \E{\DlX}}^2 \1{\frac{\abs{\DlX - \E{\DlX}}^2 }{\epsilon^2 M_\ell^2} > \nu}}\\
& = \sum_{\ell=0}^L \left\{\frac{\sqrt{V_\ell C_\ell}}{S_L} \E{\frac{\abs{\DlX - \E{\DlX}}^2}{V_\ell}
    \1{\frac{\abs{\DlX - \E{\DlX}}^2 }{V_\ell} > \frac{\epsilon^2
        M_\ell^2}{V_\ell} \nu}} \right\} - \rho(\epsilon)\;,
\end{split} 
\end{equation}
where the mapping $\rho\colon\bR_+ \to [0,\infty)$, satisfying
$\lim_{\epsilon\downarrow 0} \rho(\epsilon) =0$, can be derived as in
the proof of Lemma~\ref{lem:varEpsRelation}.  
\end{proof}

In settings with $\lim_{\ell \to \infty} S_\ell< \infty$,
the summability of the sequence
$\{\sqrt{C_\ell V_\ell}\}$ turns out to be sufficient to
prove that the extended CLT condition holds.
\begin{theorem}\label{thm:betaLargest}
  Let $\cA_{ML}$ denote the variance minimizing MLMC estimator applied to estimate
    the expectation of $X \in L^2(\Omega)$ based on the collection of
    r.v.
    $\{X_\ell\} \subset L^2(\Omega)$ satisfying
    Assumptions~\ref{ass:mlmcRates} and $\lim_{\ell \to \infty} S_\ell< \infty$.
    Then the extended CLT
  condition~\eqref{eq:extendedClt} is satisfied for the normalized
  estimator.
\end{theorem}
Note that the setting $\beta >\gamma$ is completely
covered by Theorem~\ref{thm:betaLargest}, as then 
\[
  S \defeq \lim_{k \to \infty} S_k = \lim_{k \to \infty}
  \sum_{\ell=0}^k \sqrt{V_\ell C_\ell} \leq c \lim_{k \to \infty}
  \sum_{\ell =0}^k 2^{(\gamma-\beta)\ell/2} < \infty\;.
\]

\begin{proof}
 We prove this result by verifying that
  condition~\eqref{eq:lindebergCond2} holds. 

As the sequence ${\{S_\ell\}}$ is monotonically
increasing, it is contained in the bounded interval $[S_0, S]$ with
$S_0 >0$.  Consequently, Lindeberg's
condition~\eqref{eq:lindebergCond2} is equivalent to:
\[
\lim_{\epsilon \downarrow 0} \sum_{\ell=0}^L \1{V_\ell >0} \sqrt{\frac{C_\ell}{V_\ell}}
\E{\abs{\DlX - \E{\DlX}}^2
    \1{\abs{\DlX - \E{\DlX}}^2  > \epsilon^2 M_\ell^2 \nu}} =0\;,\quad\forall\,\nu>0\;.
\]
Fix a $\nu>0$. Then for all $\ell \in \bN_0$, 
\[
\E{\abs{\DlX - \E{\DlX}}^2
  \1{\abs{\DlX - \E{\DlX}}^2  > \epsilon^2 M_\ell^2 \nu}} \le V_\ell.
\]
By the preceding inequality and the summability of the sequence $\{V_\ell C_\ell\}$,
the dominated convergence theorem yields that 
\begin{equation}\label{eq:dct1}
\begin{split}
&\lim_{\epsilon \downarrow 0} \sum_{\ell=0}^L \1{V_\ell >0} \sqrt{\frac{ C_\ell}{V_\ell}} \E{\abs{\DlX - \E{\DlX}}^2
  \1{\abs{\DlX - \E{\DlX}}^2  > \epsilon^2 M_\ell^2 \nu}}\\
&= \sum_{\ell=0}^\infty  \1{V_\ell >0} \sqrt{\frac{ C_\ell}{V_\ell}}  \lim_{\epsilon \downarrow 0} \E{\abs{\DlX - \E{\DlX}}^2
    \1{\abs{\DlX - \E{\DlX}}^2  > \epsilon^2 M_\ell^2 \nu}}.
\end{split}
\end{equation}
For all $\ell \in \bN_0$ such that $V_\ell >0$,
\[
\lim_{\epsilon \downarrow 0} \epsilon^2M_\ell^2(\epsilon) \ge
\lim_{\epsilon \downarrow 0} \epsilon^{-2} \frac{V_{\ell}}{C_\ell} S_L^2 = \infty,
\]
and the dominated convergence theorem applies for all $\ell \in \bN_0$:
\begin{equation}\label{eq:dct2}
\begin{split}
& \1{V_\ell >0} \sqrt{\frac{ C_\ell}{V_\ell}}  \lim_{\epsilon \downarrow 0} \E{\abs{\DlX - \E{\DlX}}^2
    \1{\abs{\DlX - \E{\DlX}}^2  > \epsilon^2 M_\ell^2 \nu}}\\
& = \1{V_\ell >0} \sqrt{\frac{ C_\ell}{V_\ell}}   \E{ \lim_{\epsilon \downarrow 0} \abs{\DlX - \E{\DlX}}^2
    \1{\abs{\DlX - \E{\DlX}}^2  > \epsilon^2 M_\ell^2 \nu}}\\
  & = 0.
\end{split}
\end{equation}
As the above argument is valid for any fixed $\nu>0$,
equations~\eqref{eq:dct1} and~\eqref{eq:dct2} verify that 
Lindeberg's condition holds.

\end{proof}

We next verify the extended CLT condition for the variance minimizing
MLMC estimator in settings with $\lim_{\ell \to \infty} S_\ell = \infty$.

\begin{theorem}\label{thm:gammaLargest}
  Let $\cA_{ML}$ denote the variance minimizing MLMC estimator applied to estimate
  the expectation of $X \in L^2(\Omega)$ based on the collection of
  r.v.
  $\{X_\ell\} \subset L^2(\Omega)$ satisfying
  Assumptions~\ref{ass:mlmcRates} and~\ref{ass:mlmcRates2}.
  Assume that $\lim_{\ell \to \infty} S_\ell = \infty$
  and that 
  \[
 \lim_{\ell \to \infty} \1{V_\ell >0} \E{ \frac{ \abs{\DlX  - \E{\DlX}}^2}{ V_\ell}  \1{\frac{ \abs{\DlX  -
      \E{\DlX}}^2}{V_\ell} > 2^{(2\alpha - \gamma)\ell} S_{\ell}^2\nu}} = 0
 \]
 holds for any $\nu >0$.
 Then the extended CLT
 condition~\eqref{eq:extendedClt} is satisfied for the normalized
 MLMC estimator.
\end{theorem}
\begin{proof}
  From~\eqref{eq:Ml} and $C_\ell = \Theta_\ell(2^{\gamma \ell})$ it
  follows that there exists a $c>0$ such that
  \[
  \frac{\epsilon^2 M_\ell^2}{V_\ell} \ge \frac{\epsilon^{-2}S_\ell^2}{C_\ell} > c 2^{(2\alpha - \gamma)\ell} S_{\ell}^2.
  \]
  Consequently,
  \[
  \begin{split}
  &\sum_{\ell=0}^L  \frac{\sqrt{V_\ell C_\ell}}{S_L}  \E{\frac{\abs{\DlX - \E{\DlX}}^2}{V_\ell}
    \1{\frac{\abs{\DlX - \E{\DlX}}^2 }{V_\ell} > \frac{\epsilon^2 M_\ell^2}{V_\ell} \nu}}\\
  &\le \sum_{\ell=0}^L  \frac{\sqrt{V_\ell C_\ell}}{S_L}  \E{\frac{\abs{\DlX - \E{\DlX}}^2}{V_\ell}
    \1{\frac{\abs{\DlX - \E{\DlX}}^2 }{V_\ell} >  \nu c 2^{(2\alpha - \gamma)\ell} S_{\ell}^2}}.
  \end{split}
  \]
  Let $\widetilde L : (0,\infty) \to \bN_0$ be a monotonically
  decreasing function satisfying the constraints 
  \[
  \lim_{\epsilon \downarrow 0} \widetilde{L}(\epsilon)
  = \infty \quad \text{and} \quad \lim_{\epsilon \downarrow 0} \frac{S_{\widetilde L(\epsilon)}}{S_{L(\epsilon)}}
  = 0.
  \]
  Under the current assumption $\lim_{\epsilon \downarrow 0} S_{L(\epsilon)} = \infty$,
  it is always possible to construct such an $\widetilde L$, e.g.,
  \[
  \widetilde L(\epsilon) := \min \left\{\ell \in \bN_0 \mid
  S_{\ell+1} \ge \sqrt{S_{L(\epsilon)}}\right\}.
  \]
  Provided that $\epsilon >0$ is sufficiently small,
  it holds that $\widetilde L < L$ and we may write 
  \[
  \begin{split}
    &\sum_{\ell=0}^L \frac{\sqrt{V_\ell C_\ell}}{S_L}
    \E{\frac{\abs{\DlX - \E{\DlX}}^2}{V_\ell} \1{\frac{\abs{\DlX -
            \E{\DlX}}^2 }{V_\ell} > \nu c 2^{(2\alpha - \gamma)\ell}
        S_{\ell}^2}}\\
    &\le \sum_{\ell=0}^{\widetilde L} \frac{\sqrt{V_\ell C_\ell}}{S_L}
    + \sum_{\ell=\widetilde L +1}^L \frac{\sqrt{V_\ell C_\ell}}{S_L}
    \E{\frac{\abs{\DlX - \E{\DlX}}^2}{V_\ell} \1{\frac{\abs{\DlX -
            \E{\DlX}}^2 }{V_\ell} > \nu c 2^{(2\alpha - \gamma)\ell}
        S_{\ell}^2}}\\
    &\le \frac{S_{\widetilde L}}{S_{L}} + \frac{S_L- S_{\widetilde L}}{S_{L}}
    \times \sup_{\ell > \widetilde L}
    \E{\frac{\abs{\DlX - \E{\DlX}}^2}{V_\ell} \1{\frac{\abs{\DlX -\E{\DlX}}^2 }{V_\ell} > \nu c 2^{(2\alpha - \gamma)\ell}S_\ell^2}}.
  \end{split}
  \]
  Consequently,
  \[
  \begin{split}
  & \lim_{\epsilon \downarrow 0} \sum_{\ell=0}^L  \frac{\sqrt{V_\ell C_\ell}}{S_L}  \E{\frac{\abs{\DlX - \E{\DlX}}^2}{V_\ell}
    \1{\frac{\abs{\DlX - \E{\DlX}}^2 }{V_\ell} > \frac{\epsilon^2 M_\ell^2}{V_\ell} \nu}} \\
  & \le \lim_{\epsilon \downarrow 0} \frac{S_{\widetilde L}}{S_{L}}
  + \limsup_{\ell \to \infty} \E{\frac{\abs{\DlX - \E{\DlX}}^2}{V_\ell} \1{\frac{\abs{\DlX -\E{\DlX}}^2 }{V_\ell} > \nu c 2^{(2\alpha - \gamma)\ell}S_\ell^2}}\\
  &= 0.
  \end{split}
  \]

\end{proof}






  

\subsection{CLT for the mass-shifted MLMC estimator}
\label{sec:massShiftProof}
The key feature of the mass-shifted MLMC estimator that is
particularly handy for proving the CLT is that irrespective of
whether $\{ S_{\ell}\}$ is uniformly bounded from above or not,
it will always be the case that $\lim_{\ell \to \infty} \widetilde S_{\ell} < \infty$.
The CLT follows by this property and an extension of Theorem~\ref{thm:betaLargest}.

\begin{proof}[Proof of Theorem~\ref{thm:mainResult2}]
  Recall that the mass-shifted MLMC estimator is given by
  \[
  \widetilde \cA_{ML} = \sum_{\ell=0}^{L} \sum_{i=1}^{\widetilde M_\ell} \frac{\DlX^i}{ \widetilde M_\ell},
  \]
  where $\widetilde M_\ell$ for a given $\xi >0$ is defined in equation~\eqref{eq:mlTilde}
  and $\{\Delta_\ell X\}$ is a sequence of r.v.~satisfying Assumption~\ref{ass:mlmcRates}
  for a rate triplet $\alpha,\beta,\gamma$.
  Let $\{Y_\ell\}_{\ell=-1}^\infty \subset L^2(\Omega)$ denote
  a auxiliary sequence satisfying $Y_{-1} :=0$ and for all $\ell \ge 0$,
  \[
  Y_\ell \stackrel{\rm{d}}{=} X_\ell, \qquad \Delta_\ell Y \stackrel{\rm{d}}{=} \Delta_\ell X,
  \]
  and 
  \[
  \begin{split}
  \overline C_\ell := \mathrm{Cost}(\Delta_\ell Y) = \frac{\mathrm{Cost}(\Delta_\ell X)}{(S_\ell+1)^2 \log(S_\ell + 1)^{2(1+\xi)}}
  = \frac{C_\ell }{(S_\ell+1)^2 \log(S_\ell + 1)^{2(1+\xi)}}.
  \end{split}
  \]
  Let $\cA_{ML}$ denote the variance minimizing MLMC estimator applied to $\{Y_\ell\}_{\ell=-1}^\infty$,
  i.e.,
  \begin{equation}\label{eq:aMLY}
  \cA_{ML} = \sum_{\ell=0}^L \sum_{i=1}^{M_\ell} \frac{\Delta_\ell Y^i}{M_\ell},
  \end{equation}
  where it follows by $\Var{\Delta_\ell Y} = \Var{\Delta_\ell X} = V_\ell$ and equation~\eqref{eq:Ml} that
  \[
  M_\ell = \max\pr{ \ceil{\epsilon^{-2} \sqrt{\frac{V_\ell}{\overline C_\ell}} \sum_{\ell=0}^L \sqrt{V_\ell \overline C_\ell} }, 1 }.
  \]
  By construction,
  \[
  \sum_{\ell=0}^L \sqrt{V_\ell \overline C_\ell} = \widetilde  S_L, 
  \]
  hence, $M_\ell = \widetilde M_\ell$ for all $\ell \in [0, L]$.
  Consequently, $\cA_{ML} \stackrel{\rm{d}}{=} \widetilde \cA_{ML}$, so the
  theorem follows if we can prove the CLT for the normalized version of $\cA_{ML}$.
      
  The collection of random variables
  $\{Y_\ell\}$ satisfies the following slightly altered version
  of Assumption~\ref{ass:mlmcRates} (where $\Theta_\ell(2^{\gamma \ell})$ is replaced
  by $\cO_\ell(2^{\bar \gamma \ell})$ in condition (iii)):
\begin{itemize}
 
 \item[(i)] for some $c_\alpha>0$, $\abs{ \E{X - Y_\ell} }   \le c_\alpha 2^{-\alpha \ell}$ for all $\ell\ge0$,
 
 \item[(ii)] $\text{Var}(\Delta_\ell Y) = \cO_\ell(2^{-\beta \ell})$,
   
 \item[(iii)] $\overline{C}_\ell = \cO_\ell(2^{\bar \gamma \ell})$ and $\inf_{\ell \in \bN_0} \overline{C}_\ell > c>0$,
\end{itemize}
where $\bar \gamma \in (0, \gamma ]$ and 
$\alpha,\beta,\gamma >0$ (as everywhere else in this proof)
stems from the rate triplet of $\{X_\ell\}$.
Moreover,
\[
\min(\beta,\gamma) \le 2 \alpha \implies \min(\beta,\bar \gamma) \le 2 \alpha,
\]
and since $\{S_\ell\}$ is monotonically increasing,
  \[
  \begin{split}
    \widetilde  S_L &= \sum_{\ell=0}^L \frac{\sqrt{V_\ell  C_\ell} }{(S_\ell+1) \log(S_\ell + 1)^{1+\xi}}\\
    & =  \sum_{\ell=0}^L \frac{S_\ell - S_{\ell-1} }{(S_\ell+1) \log(S_\ell + 1)^{1+\xi}} \\
    & \le   \int_{S_0}^{S_L} \frac{1}{ (s+1)\log(s+1)^{1+\xi}} ds \\
    & < \frac{1}{\xi\log(S_0+1)^{\xi}} < \infty.
    \end{split}
  \]
This shows that $\widetilde S_\ell \in [\widetilde S_0, \widetilde S]$
for all $\ell \ge 0$, 
where $\widetilde S_0 = V_0 \overline C_0>0$
and $\widetilde S =  \lim_{\ell \to \infty} \widetilde S_\ell<\infty$.
Using the uniform bounds on $\{\widetilde S_\ell\}$ and the properties
of the rate triplet for $\{Y_\ell\}$,
the proofs of Lemma~\ref{lem:varEpsRelation}, Corollary~\ref{cor:extCLT}
and Theorem~\ref{thm:betaLargest} straightforwardly extends
to the current setting, verifying the CLT for
the normalized version of the estimator~\eqref{eq:aMLY}. 
\end{proof}


\section*{References}

\begin{thebibliography}{100}
\expandafter\ifx\csname url\endcsname\relax
  \def\url#1{\texttt{#1}}\fi
\expandafter\ifx\csname urlprefix\endcsname\relax\def\urlprefix{URL }\fi
\expandafter\ifx\csname href\endcsname\relax
  \def\href#1#2{#2} \def\path#1{#1}\fi

\bibitem{MR3297771}
  M.~Ben~Alaya, A.~Kebaier.
  \newblock {Central
    limit theorem for the multilevel {M}onte {C}arlo {E}uler method}.
  \newblock {\em Ann. Appl. Probab.} 25~(1) (2015) 211--234.


\bibitem{anderson}
D.~F.~Anderson and D.~J.~Higham.
\newblock Multilevel monte carlo for continuous time markov chains, with
  applications in biochemical kinetics.
\newblock {\em Multiscale Modeling \& Simulation}, 10(1):146--179, 2012.

\bibitem{BallesioEtAl}
  M.~Ballesio, J.~Beck, A.~Pandey, L.~Parisi, E.~{von Schwerin}, R.~Tempone.
  \newblock {Multilevel Monte Carlo Acceleration of Seismic Wave Propagation under Uncertainty}.
  \newblock {\em arXiv:1810.01710}, 2018..

\bibitem{barthLang}
A.~Barth and A.~Lang.
\newblock Multilevel monte carlo method with applications to stochastic partial
  differential equations.
\newblock {\em International Journal of Computer Mathematics},
  89(18):2479--2498, 2012.

\bibitem{barthLang2}
A.~Barth, A.~Lang, and C.~Schwab.
\newblock Multilevel monte carlo method for parabolic stochastic partial
  differential equations.
\newblock {\em BIT Numerical Mathematics}, 53(1):3--27, 2013.

\bibitem{MR3501366} C.~Bayer, P.~K.~Friz, S.~Riedel, J.~Schoenmakers.
\newblock {From rough path estimates to multilevel {M}onte {C}arlo}.
\newblock {\em SIAM J. Numer. Anal.} 54~(3) (2016) 1449--1483.


\bibitem{BeckEtAl}
  J.~Beck, B.~M.~Dia, L.~FR~Espath, R.~Tempone.
  \newblock {Multilevel Double Loop Monte Carlo and Stochastic Collocation Methods with Importance Sampling for Bayesian Optimal Experimental Design}.
 \newblock {\em arXiv:1811.11469}, 2018.
  

\bibitem{lawHoel}
A.~Chernov, H.~Hoel, K.~JH~Law, F.~Nobile, and R.~Tempone.
\newblock Multilevel ensemble kalman filtering for spatially extended models.
\newblock {\em arXiv:1608.08558}, 2016.

\bibitem{MR3348197}
N.~Collier, A.-L. Haji-Ali, F.~Nobile, E.~von Schwerin, R.~Tempone.
\newblock {A continuation multilevel {M}onte {C}arlo algorithm}.
\newblock {\em BIT} 55~(2) (2015) 399--432.



\bibitem{MR3449315}
  S.~Dereich, S.~Li.
  \newblock {Multilevel {M}onte {C}arlo for {L}\'evy-driven {SDE}s: central limit theorems for
    adaptive {E}uler schemes}.
  \newblock {\em Ann. Appl. Probab.} 26~(1) (2016) 136--185.


\bibitem{scheichlMcmc}
T.~J.~Dodwell, C.~Ketelsen, R.~Scheichl, and A.~L.~Teckentrup.
\newblock A hierarchical multilevel markov chain monte carlo algorithm with
  applications to uncertainty quantification in subsurface flow.
\newblock {\em SIAM/ASA Journal on Uncertainty Quantification},
  3(1):1075--1108, 2015.

\bibitem{MR1609153}
  R.~Durrett.
  \newblock {Probability: theory and examples.}
  \newblock  2nd Edition, Duxbury Press,  Belmont, CA, 1996.

\bibitem{MR2436856}
M.~B.~Giles.
\newblock Multilevel {M}onte {C}arlo path simulation.
\newblock {\em Oper. Res.}, 56(3):607--617, 2008.

\bibitem{MR3349310}
M~B.~Giles.
\newblock Multilevel {M}onte {C}arlo methods.
\newblock {\em Acta Numer.}, 24:259--328, 2015.



\bibitem{Giorgi}
D.~Giorgi, V.~Lemaire, G.~Pag\`es.
\newblock {Limit theorems for weighted and regular multilevel estimators.}
  \newblock {\em Monte Carlo Methods Appl.} 23~(1) (2017) 43--70.

\bibitem{reich}
A.~Gregory, C.~J.~Cotter, and S.~Reich.
\newblock Multilevel ensemble transform particle filtering.
\newblock {\em SIAM Journal on Scientific Computing}, 38(3):A1317--A1338, 2016.

  
\bibitem{hoel2016}
H.~Hoel, J.~H{\"a}pp{\"o}l{\"a}, and R.~Tempone.
\newblock Construction of a mean square error adaptive euler--maruyama method
  with applications in multilevel Monte Carlo.
\newblock In {\em Monte Carlo and Quasi-Monte Carlo Methods}, p.~29--86.
  Springer, 2016.



\bibitem{hoel2014implementation}
  H.~Hoel, E.~Von~Schwerin, A.~Szepessy, R.~Tempone.
\newblock Implementation and analysis of an adaptive multilevel monte carlo
  algorithm.
\newblock {\em Monte Carlo Methods and Applications}, 20(1):1--41, 2014.

  
\bibitem{MR1629093}
S.~Heinrich.
\newblock Monte {C}arlo complexity of global solution of integral equations.
\newblock {\em J. Complexity}, 14(2):151--175, 1998.

\bibitem{hoang}
V.~H.~Hoang, C.~Schwab, and A.~M.~Stuart.
\newblock Complexity analysis of accelerated mcmc methods for bayesian
  inversion.
\newblock {\em Inverse Problems}, 29(8):085010, 2013.

\bibitem{lawJasra}
A.~Jasra, K.~Kamatani, K.~JH Law, and Y.~Zhou.
\newblock Multilevel particle filters.
\newblock {\em SIAM Journal on Numerical Analysis}, 55(6):3068--3096, 2017.

\bibitem{kebaierImpSampling}
A.~Kebaier and J.~Lelong.
\newblock Coupling importance sampling and multilevel monte carlo using sample
  average approximation.
\newblock {\em Methodology and Computing in Applied Probability},
  20(2):611--641, 2018.

  
\bibitem{klenke}
  A.~Klenke.
  \newblock {Probability theory.}
  \newblock  2nd Edition, {\em Universitext}, Springer, London, 2014.


\bibitem{kloedenPlaten}
Peter~E. Kloeden and Eckhard Platen.
\newblock {\em Numerical solution of stochastic differential equations},
  volume~23 of {\em Applications of Mathematics (New York)}.
\newblock Springer-Verlag, Berlin, 1992.

  
\bibitem{ullmannBayes}
J.~Latz, I.~Papaioannou, and E.~Ullmann.
\newblock {Multilevel sequential monte carlo for bayesian inverse problems.}
\newblock {\em Journal of Computational Physics}, 368:154--178, 2018.

\bibitem{mishra}
S.~Mishra and C.~Schwab.
\newblock Monte-carlo finite-volume methods in uncertainty quantification for
  hyperbolic conservation laws.
\newblock In {\em Uncertainty Quantification for Hyperbolic and Kinetic
  Equations}, pages 231--277. Springer, 2017.

\bibitem{moraes}
A.~Moraes, R.~Tempone, and P.~Vilanova.
\newblock Multilevel hybrid chernoff tau-leap.
\newblock {\em BIT Numerical Mathematics}, 56(1):189--239, 2016.




\bibitem{glynnRhee2015}
C.-han~Rhee and P.~W.~Glynn.
\newblock Unbiased estimation with square root convergence for sde models.
\newblock {\em Operations Research}, 63(5):1026--1043, 2015.



  
\bibitem{Teckentrup2013}
  A.~L.~Teckentrup, R.~Scheichl, M.~B.~Giles, E.~Ullmann.
\newblock {Further analysis of multilevel {M}onte {C}arlo methods for elliptic {PDE}s with random coefficients},
\newblock {\em Numer. Math.} 125~(3) (2013) 569--600.



\bibitem{ullmannRare}
E.~Ullmann and I.~Papaioannou.
\newblock Multilevel estimation of rare events.
\newblock {\em SIAM/ASA Journal on Uncertainty Quantification}, 3(1):922--953,
  2015.


\bibitem{glynnZheng2016}
Z.~Zheng, J.~Blanchet, and P.~W.~Glynn.
\newblock Rates of convergence and clts for subcanonical debiased mlmc.
\newblock In {\em International Conference on Monte Carlo and Quasi-Monte Carlo
  Methods in Scientific Computing}, pages 465--479. Springer, 2016.

\bibitem{glynnZheng2017}
Z.~Zheng and P.~W.~Glynn.
\newblock {A CLT for infinitely stratified estimators, with applications to
  debiased mlmc.}
\newblock {\em ESAIM: Proceedings and Surveys}, 59:104--114, 2017.






\end{thebibliography}

\end{document}